\documentclass[a4paper,10pt]{article}
\pdfoutput=1

\usepackage{ae,aecompl}
\usepackage[T1]{fontenc}
\usepackage[utf8]{inputenc}
\usepackage{amssymb,amsmath,amsthm}
	\theoremstyle{plain}
		\newtheorem{thm}{Theorem}[section]
		\newtheorem{lem}[thm]{Lemma}
		\newtheorem{prop}[thm]{Proposition}
		
		\newtheorem{question}[thm]{Question}
	\theoremstyle{definition}
		\newtheorem{defn}[thm]{Definition}

	\theoremstyle{remark}
		\newtheorem{rem}[thm]{Remark}
		
\usepackage{mathtools}
\usepackage[only,Ydown]{stmaryrd}
\usepackage{xifthen}
\usepackage[backend=biber,style=alphabetic,sorting=nty]{biblatex}
\usepackage[margin=2cm]{geometry}
\usepackage{fancyhdr}
	\newcommand{\Author}{Dávid R. Szabó}
	\newcommand{\Title}{Special $p$-groups acting on compact manifolds}
\PassOptionsToPackage{hidelinks}{hyperref} 
	\newcommand{\email}[1]{\href{mailto:#1}{\nolinkurl{#1}}}
\usepackage{tikz-cd}
\usepackage{todonotes}

\newcommand{\Np}{\mathbb{N}_{+}}
\newcommand{\N}{\mathbb{N}_{0}}
\newcommand{\Z}{\mathbb{Z}}
\newcommand{\Q}{\mathbb{Q}}
\newcommand{\R}{\mathbb{R}}
\newcommand{\C}{\mathbb{C}}

\renewcommand{\phi}{\varphi}
\renewcommand{\rho}{\varrho}
\newcommand{\Sp}{\mathcal{S}}
\newcommand{\E}{\mathcal{E}}
\newcommand{\A}{\mathcal{A}}
\newcommand{\cat}{\mathcal{C}}

\DeclareMathOperator{\GL}{GL}
\DeclareMathOperator{\SU}{SU}
\DeclareMathOperator{\U}{U}
\DeclareMathOperator{\Diff}{Diff}
\DeclareMathOperator{\Aut}{Aut}
\DeclareMathOperator{\Biholo}{Bih}
\DeclareMathOperator{\Hom}{Hom}

\newcommand{\sphere}[1]{\mathbb{S}^{#1}}
\newcommand{\torus}[1]{{\mathbb{T}^{#1}}}
\newcommand{\ctorus}[1]{{\mathbb{T}_\C^{#1}}}

\DeclareMathOperator{\ch}{ch}
\DeclareMathOperator{\im}{Im}
\newcommand{\baseAction}[1]{\overline{#1}}

\newcommand{\isom}{\cong}

\newcommand{\acts}{\circlearrowright}
\newcommand{\tensor}{\otimes}
\newcommand{\extTensor}{\operatorname{\boxtimes}}
\newcommand{\cprod}[1][]{\mathbin{\Ydown{\ifthenelse{\isempty{#1}}{}{\hspace{-2px}_{#1}}}}}
\newcommand{\cuptimes}{\mathbin{\smile}}

\newenvironment{summary}{\begin{quote}\small}{\end{quote}}
\newenvironment{smallpmatrix}{\left(\begin{smallmatrix}}{\end{smallmatrix}\right)}

\title{\Title}
\author{\Author\footnote{
The project leading to this application has received funding from the European Research Council (ERC) under the European Union's Horizon 2020 research and innovation programme (grant agreement No 741420). 
The author has received partial funding from the Hungarian Scientific Research Fund (OTKA) K115799.}}

\begin{document}
\maketitle
\begin{abstract}
\citeauthor{Riera} proved at \cite{Riera} that the diffeomorphism group of particular compact manifolds are not Jordan by exhibiting subgroups isomorphic to extra-special $p$-groups of exponent $p$ for primes $p$ satisfying some conditions. Generalising the methods of that paper, we construct a compact connected smooth real manifold for every natural number $r$ whose diffeomorphism group contains not only every extra-special $p$-group, but also every special $p$-group of order $p^r$ independently of its exponent for every prime $p$. We obtain a similar statement about finite Heisenberg groups as well as we display a very explicit counterexample to the conjecture of Ghys about Jordan property of diffeomorphism group of compact manifolds.
\end{abstract}

\section{Introduction}\label{sec:intro}
One approach to study (infinite) groups is via their finite subgroups. 
A group $G$ is called a \emph{Jordan group} if there exists a positive integer $J_G$, depending only on $G$, such that every finite subgroup $K$ of $G$ contains a normal abelian subgroup whose index in $K$ is at most $J_G$ \cite[Definition~2.1]{Popov}. 
This definition was motivated by the classical result of \citeauthor{Jordan1877} from \citedate{Jordan1877} stating that the group $\GL_n(\C)$ is Jordan \cite{Jordan1877}. 
Determining if the Jordan property holds for the group of (birational) automorphisms of varieties or for the group of diffeomorphisms of manifolds is an active area with many affirmative answers and counterexamples.
Ghys conjectured that the diffeomorphism group of every \emph{compact} smooth manifold is Jordan \cite[\S13.1]{FisherSurvey2008}. 
This was shown to be true for compact manifolds of dimension at most $3$  \cite[Theorem~1]{Zimmermann2014}, 
for acyclic manifolds, 
for connected compact manifolds with non-zero Euler characteristic, 
and for integral homology spheres \cite[Theorem~1.2]{Riera2018}. 
Other examples can be found at \cite[\S1]{CsPSz} and \cite[\S1.1]{Riera2018}.

On the other hand, \citeauthor{zarhin2014} gave the first example of an algebraic variety whose birational automorphism group is not Jordan, namely the direct product of an elliptic curve and the projective line \cite[Theorem 1.2.]{zarhin2014}. To show the failure of the Jordan property, \citeauthor{zarhin2014} used actions of certain analogues of the Heisenberg group introduced by Mumford in 1966. 
Since then, these groups have served as the base of mainly all constructions producing non-Jordan groups. 
A prominent example is the first counterexample to Ghys' conjecture: in \citeyear{CsPSz}, \citeauthor{CsPSz} showed that the diffeomorphism group of the compact $4$-manifold $\torus{2}\times \sphere{2}$ is not Jordan by embedding an infinite family of $3$-dimensional finite Heisenberg groups satisfying certain conditions to the diffeomorphism group \cite[Theorem~1]{CsPSz}.
Based on this paper, \citeauthor{Riera} found many other counterexamples by embedding certain families of higher dimensional Heisenberg groups over rings of prime order satisfying some number theoretical conditions to  diffeomorphism groups of certain compact manifolds \cite[Theorem~1.1]{Riera}.
These counterexamples to Ghys' conjecture raise the following question. 
\begin{question}\label{q:main}
For which families $\mathcal{F}$ of finite groups does there exist a compact manifold $M$ such that $G\hookrightarrow \Diff(M)$ for every $G\in\mathcal{F}$?
\end{question}
Dropping the requirements on the manifold, the problem is answered fully by \citeauthor{Popov2013}: there exists of a non-compact $4$-manifold whose diffeomorphism group contains an isomorphic copy of every finite(ly presented) group \cite[Corollary~1]{Popov2013}. 
However, in the compact case, we cannot hope for such an answer, because for example the rank of an elementary abelian $p$-group acting on a compact connected manifold is bounded by the dimension of the manifold \cite[Theorem~2.5]{Mann1963}. 
The question was also raised by \citeauthor{Riera} in September 2018 for the family consisting of all $2$-nilpotent $p$-groups of order $p^r$ for fixed $r$. While this is still open, in this paper we answer the question affirmatively for two subfamilies where the index of maximal normal abelian subgroup is not bounded: finite Heisenberg groups (Theorem~\ref{thm:Heisenberg}) and special $p$-groups (Theorem~\ref{thm:specialGroups}) using ideas from \cite{Riera}.

\begin{defn}\label{def:Heisenberg}
For any natural $n$, define the \emph{Heisenberg group} $H_{2n+1}(R)$ of dimension $2n+1$ over a ring $R$ to be the multiplicative matrix group 
\begin{equation}\label{eq:Heisenberg}
H_{2n+1}(R):=\left\{
\begin{smallpmatrix}
1 & a & c\\ 
0 & I_n & b \\ 
0 & 0 & 1
\end{smallpmatrix}
:
a\in R^{1\times n}, b\in R^{n\times 1}, c\in R 
\right\}	
\end{equation}
where $I_n\in R^{n\times n}$ is the identity matrix.
Define the family $\mathcal{H}_{2n+1}:=\{H_{2n+1}(\Z/(d)):d\in \Np\}$.
\end{defn}

\begin{thm}\label{thm:Heisenberg}
Question~\ref{q:main} has an affirmative answer for the family $\mathcal{H}_{2n+1}$ for every natural number $n$, i.e. there is a smooth compact connected manifold having a faithful action via diffeomorphisms of every  Heisenberg group $H_{2n+1}(\Z/(d))$.
\end{thm}
\begin{rem}
We show that $\torus{2n}\times \U(5^n n!)$ -- the direct product of a torus and a unitary group -- is such a manifold. 
Since $H_{2n+1}(\Z/(p))$ contains an elementary $p$-group of rank $n$, no similar statement can hold for $\bigcup_{n=1}^\infty \mathcal{H}_{2n+1}$, so the restriction on the dimension is necessary. 
As noted above, \cite{CsPSz} proved Theorem~\ref{thm:Heisenberg} in the case $n=1$ for infinitely many $d$'s of the same parity, 
while \cite{Riera} proved it for general $n$ but only for  $H_{2n+1}(\Z/(p))\in\mathcal{H}_{2n+1}$ for big enough primes $p$ satisfying $p\equiv 1\pmod{n+1}$. 
\end{rem}

The other result of this paper is about a different generalisation of the groups  $H_3(\Z/(p))$ for prime $p$.
\begin{defn}\label{def:specialGroups}
For a prime number $p$, a finite $p$-group $S$ is said to be \emph{special}, if either $S$ is elementary abelian 
or if its Frattini subgroup $\Phi(S)$, 
centre $Z(S)$ and commutator subgroup $[S,S]$ coincide and $\Phi(S)$ is  elementary abelian. 
A non-abelian special $p$-group with cyclic centre (i.e. of rank $1$) is called \emph{extra-special}. 
This is a large natural class of $p$-groups whose importance is explained for example at \cite[\S4]{Suzuki2}. 
Define the family $\mathcal{P}_r:=\{S: S \text{ special $p$-group}, p\text{ prime}, |S|=p^r\}$ for every natural number $r$.
\end{defn}

\begin{thm}\label{thm:specialGroups}
Question~\ref{q:main} has an affirmative answer for the family $\mathcal{P}_r$ for every natural number $r$, i.e. there is a smooth compact connected manifold having a faithful action via diffeomorphisms of every special $p$-group of order $p^r$ for every prime $p$.
\end{thm}
\begin{rem}
The proof is constructive and a simple concrete example of $\torus{r^2+r}\times \U(5^r \lfloor r/2\rfloor !)^r$ can be obtained. 
The necessary rank-dependency from \cite[Theorem~2.5]{Mann1963} prohibits considering the set $\bigcup_{r=1}^\infty \mathcal{P}_r$ of all special $p$-groups, and is encoded in the condition on the order.
In fact, Theorems~\ref{thm:Heisenberg} and \ref{thm:specialGroups} are both special cases of Theorem~\ref{thm:main} that we state and prove at \S\ref{sec:main}. 
With the terminology above, \cite{zarhin2014}, \cite{CsPSz} and \cite{Riera} have embedded special $p$-groups of rank $1$ and exponent $p$ with various restrictions on $p$. Theorem~\ref{thm:specialGroups} shows that every condition on the exponent and on $p$ can be dropped while altering the rank arbitrarily.
\end{rem}

\paragraph*{Acknowledgement} The author is thankful to László Pyber for the introduction to this topic, for the group theoretical hints and for fruitful question, to Endre Szabó for the detailed discussions and insights about the  constructions and to András Szűcs for the discussion about the Chern character of vector bundles over torus.

\paragraph*{Notations, conventions}
In this paper, the following notation is used.
$\Np$ is the set of positive integers, $\N=\{0\}\cup \Np$, $\sphere{1}=\{z\in\C:|z|=1\}$ is the $1$-sphere. 
For $n\in\Np$,
$\torus{n}:=(\sphere{1})^n$ is the real $n$-torus,
$\ctorus{n}=(\C/\Lambda)^n$ is the complex torus of complex dimension $n$ for $\Lambda:=\Z+i\Z\subset \C$, 
$\mu_n:=e^{2\pi i/k}\in\sphere{1}$ is a primitive $n$th root of unity, 
$\Z_n:=\Z/n\Z$ is the cyclic group of order $n$, 
$\SU(n)\subseteq \U(n)\subseteq \C^{n\times n}$ the compact group of $n\times n$ special unitary, respectively unitary matrices. 
For a ring $R$, 
$R^\times$ denotes the multiplicative group of invertible elements of a ring $R$. 
For a topological space $X$, 
$K^0(X)$ is the complex $K$-theory of $X$,
$H^\bullet(X,R):=\bigoplus_{k\in \N}H^k (X,R)$ is the cohomology ring with coefficients from $R$ (with $\cuptimes$-product as multiplication). 
We embed $H^\bullet(X,\Z)\subseteq H^\bullet(X,\Q)$ in the natural way. 
The Chern character is denoted by $\ch:K^0(X)\to H^\bullet(X,\Q)$. 
Denote the automorphism group of the object $X$ in the category of differentiable real manifolds by $\Diff(X)$ and of complex manifolds by $\Biholo(X)$.
For binary operations $\cprod[k]$, by $G_0\cprod[1] G_1 \cprod[2] \dots \cprod[n] G_n$ we mean the following order $(\dots((G_0\cprod[1] G_1) \cprod[2] G_2) \dots) \cprod[n] G_n$. 
We apply maps from the left, denote by $f_1\times f_2$ the direct product of two maps, and apply group commutators in the following order: $[g,h]:=g^{-1}h^{-1}gh$.

\section{Overview}\label{sec:overview}
\begin{summary}
In this section, we overview the main steps of the proof and the structure of the paper. The proof uses two key ingredients. 
The first one (Lemma~\ref{lem:main}) reduces the problem to vector bundles over a fixed base space. The second one is a structural description of the groups involved (Lemma~\ref{lem:Heisenberg}/\ref{lem:Suzuki}) and reduces both problems further to a smaller well-described family $\Sp_{n,m}$ of finite groups.
\end{summary}

\begin{defn}\label{def:bundleAction}
Let $\cat$ be the category of smooth real manifolds or that of complex manifolds. 
Let $\pi\in\Hom_\cat(E,X)$ be a complex vector bundle. 
Let $\Aut_\cat^\pi(E)$ be subgroup consisting of elements $\phi\in \Aut_\cat(E)$ for which there exists (a unique) $\phi^\pi\in\Aut_\cat(X)$ such that $\phi^\pi\circ \pi = \pi\circ \phi$ and $\phi$ acts $\C$-linearly on the fibres of $\pi$. 
We say group $G$ \emph{acts} on $\pi$ \emph{by vector bundle morphisms}, and write $\rho:G\acts\pi$, if there are group morphisms $\rho:G\to\Aut_\cat^\pi(E)$ and $\baseAction{\rho}:G\to\Aut_\cat(X)$ such that $(\rho(g))^\pi = \baseAction{\rho}(g)$ for every $g\in G$. In this case, we say that $\rho$ \emph{lifts} $\baseAction{\rho}$.
\end{defn}

The first key ingredient is extracted from the main flow of \cite{Riera}. This statement reduces the problem of finding actions of a fixed compact manifold, to that of vector bundles requiring a priori to have trivial Chern character, fixed rank and fixed base space.
\begin{lem}[\cite{Riera}]\label{lem:main}
Let $X$ be a smooth real manifold of dimension $N$ such that $K^0(X)$ is torsion free.
Let $\pi:E\to X$ be a complex vector bundle of rank $r\geq N/2$ with $\ch(\pi)\in H^0(X,\Q)$. 
Let a finite group $G$ act faithfully on $\pi$ by vector bundle morphisms. 
Then there is a faithful action of $G$ on the compact space $X\times \U(r)$ by diffeomorphisms.
\end{lem}
\begin{proof}
First we claim that $\pi$ is the trivial vector bundle. 
Indeed, 
note that the Chern character $\ch:K^0(X)\to H^\bullet(X,\Q)$ is injective by using \cite[Corollary \S2.4]{AtiyahHirzebruch}. 
Let $\theta$ the trivial line bundle over $X$. 
Then $\ch([\pi]-[\theta^{\oplus r}]) = 0\in H^\bullet(X,\Q)$ by assumption. 
Thus $[\pi]-[\theta^{\oplus r}]=0\in K(X)$, so by \cite[Proposition~2.1.5(ii)]{park2008complex}, there is $m\in\N$ such that $\pi\oplus \theta^{\oplus m}\isom \theta^{\oplus(r+m)}$. Finally, since $r\geq N/2$, \cite[Part II, \S9.1, Theorem~1.5]{Husemoller} implies that $\pi\isom \theta^{\oplus r}$ is trivial as claimed.

$\pi$ has a Hermitian metric $h_0$ by 	\cite[Corollary~1.7.10	]{park2008complex}. 
Taking the average, we obtain a $G$-invariant metric $h$ on $\pi$ defined by 
$h(v_1,v_2):=\frac{1}{|G|}\sum_{g\in G} h_0(g\cdot v_1,g\cdot v_2)$
for any $v_1,v_2\in\pi^{-1}(x)$ for $x\in X$. 
Let $\langle \cdot,\cdot\rangle$ be the usual Hermitian inner product on $\C^r$ and consider the 
$h$-unitary frame bundle of $\pi$, 
$F_U:=\coprod_{x\in X}\{(\C^r,\langle \cdot,\cdot\rangle)\to (\pi^{-1}(x),h) \text{ linear isometry}\}\to X$.
Since $G$ leaves $h$ invariant, $G$ acts on $F_U$ by precomposition. This action is faithful because the one on $\pi$ is so.
$F_U$ is a principal $\U(r)$-bundle, but since $\pi$ is trivial by the claim, so is $F_U$. Hence $F_U=X\times \U(r)$ has a faithful $G$-action as claimed.
\end{proof}
\begin{rem}
An analogous statement with $X\times \SU(r+1)$ can be proved in the exact same way using the injection $\U(r)\to \SU(r+1)$ without further calculation or assumptions on $G$, c.f. \cite[Lemma~2.2, and \S3]{Riera}.
\end{rem}

The second key ingredient is a structural description of families in question. 
Both $\mathcal{H}_{2n+1}$ and $\mathcal{P}_r$ can be obtained from a list of finite groups using central product of groups recursively.

\begin{defn}\label{def:EdAmd}
Let $(m,d)\in \Np^2$.
For any $0\leq j<d$ integer, define the following  groups:
\begin{equation}
E(d,j):=\langle \alpha,\beta,\gamma:
\alpha^d=\beta^d=\gamma^j,
\gamma^d=[\alpha,\gamma]=[\beta,\gamma]=1,
[\alpha,\beta]=\gamma \rangle
\label{eq:Edi}
\end{equation}
and let $\E_d:=\{E(d,j):0\leq j<d\}$. 
Denote by $\A_{m,d}$ the set of finite abelian groups $A$ whose minimal generating set is of size $m$ and has an element of order $d$.
\end{defn}
\begin{rem}\label{rem:HeisenbergGroups}
$E(d,0)$ is the Heisenberg-group of order $d^3$ of dimension $3$ over $\Z/(d)$. For $d=p$ prime, $E(p,0)$,  $E(p,1)$ is the complete list of extra-special $p$-groups of order $p^3$ (of exponent $p$ and $p^2$ respectively). 
\end{rem}

\begin{defn}[{\cite[\S2.4, pp.~137--140]{Suzuki1}}]\label{def:cprod}
Let $G_1$ and $G_2$ be groups. 
For an isomorphism $\phi:D_1\to D_2$ between central subgroups $D_i\subseteq Z(G_i)$, 
define a normal subgroup $K_\phi:=\{(z,\phi(z)^{-1}):z\in D_1\}\subseteq G_1\times G_2$. 
Then the quotient $G_1\cprod[\phi]G_2:=G_1\times G_2/K_\phi$ is called the \emph{(external) central product} of $G_1$ and $G_2$ along $\phi$ amalgamating the  central subgroups $D_1$ and $D_2$. 
Call $\phi$ a \emph{maximal central isomorphism}, if the only isomorphism between central subgroups of $G_1$ and $G_2$ extending $\phi$ is $\phi$ itself.
\end{defn}

\begin{defn}
For $(n,m)\in\Np^2$, let $\Sp_{n,m}$ be the set of finite groups $G$ which can be written as 
\begin{equation}
G\isom E_1\cprod[\phi_1] E_2 \cprod[\phi_2]\dots \cprod[\phi_{n-1}] E_n \cprod[\phi_n] A
\label{eq:typeS}
\end{equation}
for some $(E_1,\dots,E_n, A)\in (\E_d)^n \times \A_{m,d}$ for some $d\in\Np$ 
where all of the underlying amalgamations $\phi_i$ are maximal central isomorphism (c.f. Definition~\ref{def:cprod}), namely they identify cyclic subgroups of order $d$.
Finally, for $r\in \Np$ and a finite subset $I\subset \Np^2$,  define a set of finite groups by 
\begin{equation}
\Sp_I^r:= \Big\{S\leq \prod_{i=1}^r G_i: G_i\in\bigcup_{\mathclap{(n,m)\in I}}\Sp_{n,m},\, 1\leq i\leq r\Big\}.
\label{eq:SI}
\end{equation}
\end{defn}

The next structural results enables to restrict our attention to the family $\Sp_{n,m}$, because the direct product of actions of such groups yields action of $\Sp_I^r$ as detailed at Theorem~\ref{thm:main}.

\begin{lem}\label{lem:Heisenberg}
For every $n\in\Np$, $\mathcal{H}_{2n+1}\subseteq \Sp_{I}^1$  where  $I:=\{(n,1)\}$  (c.f. \eqref{eq:SI} and Definition~\ref{def:Heisenberg}).
\end{lem}
\begin{proof}
Pick $d\in\Np$ arbitrarily and define matrices from $H_{2n+1}(\Z/(d))$ that have only a single non-zero entry (apart from the main diagonal): 
for $1\leq k\leq n$, 
let $\alpha_k\in H_{2n+1}$ have $1$ at the $(1,k+1)$-entry, 
let $\beta_k\in H_{2n+1}$ have $1$ at the $(k+1,n+2)$-entry, 
and let $\gamma\in H_{2n+1}$ have $1$ at the $(1,n+2)$-entry. 
Let $E_k:=\langle \alpha_k,\beta_k, \gamma \rangle \subseteq H_{2n+1}(\Z/(d))$, and let $A:=\langle \gamma\rangle$. 
Short calculation shows that 
$E_1E_2\dots E_n A=H_{2n+1}(\Z/(d))$ and $[E_k,E_l]=[E_k,A]=1$ for $k\neq l$, so by \cite[(4.16)]{Suzuki1}, $H_{2n+1}(\Z/(d))$ is an (internal) central product of $E_1,\dots,E_n,A$. Since $E_k\isom E(d,0)$ from \eqref{eq:Edi} and for any $k\neq l$ $E_k\cap E_l=E_k\cap A=A=Z(E_k)\isom \Z_d$, we have $H_{2n+1}(\Z_d)\in \Sp_{n,1} \subseteq \Sp_{I}^1$.
\end{proof}

\begin{lem}[\citeauthor{Suzuki2}]\label{lem:Suzuki}
For every $r\in\Np$, $\mathcal{P}_r\subseteq \Sp_I^r$  where $I:=\{(n,m)\in\Np^2:2n+m\leq \max(r,3)\}$ (c.f. \eqref{eq:SI} and Definition~\ref{def:specialGroups}).
\end{lem}
\begin{proof}
Let $S$ be a special $p$-group of order $p^r$ whose centre is of size $|Z(S)|=p^s$. 
It is enough to exhibit groups $G_i\in\Sp_{n_i,m_i}$ for $1\leq i\leq s$ for some $(n_i,m_i)\in\Np^2$ such that $S\leq \prod_{i=1}^s G_i$ and $2n_i+m_i\leq \max(r,3)$.

If $S$ is elementary abelian, i.e. $S\isom \Z_p^{s}$, let $G_i:=E(p,0)\cprod[\phi] \Z_p\isom E(p,0)\in\Sp_{1,1}=:\Sp_{n_i,m_i}$ where $\phi$ is any isomorphism between the centres. Since $\Z_p\isom Z(G_i)\leq G_i$ and $2n_i+m_i=3$, this case is done.

Otherwise $S$ is a non-abelian group. Then by \cite[proof of Theorem~4.16(ii), Theorem~4.18]{Suzuki2}, there are central subgroups $K_i\subseteq Z(S)$ such that $G_i:=S/K_i$ are of the form \eqref{eq:typeS} for $(E_1,\dots,E_n,A)\in \{E(p,0), E(p,1)\}^{n_i}\times \{\Z_p^{m_i}, \Z_p^{m_i-1}\times \Z_{p^2}\} \subseteq (\E_p)^{n_i} \times \A_{m_i,p}$ for some $(n_i,m_i)\in\Np^2$ for $1\leq i\leq s$.
In particular, $G_i\in \Sp_{n_i,m_i}$ from \eqref{eq:typeS}. Finally we only need to show that $2n_i+m_i\leq \max(r,3)$ holds which is done using a simple counting argument.
We know that $|E(p,i)|=p^3$, $|A|\geq p^{m_i}$ for any $p$-group $A\in \A_{m_i,p}$ because $m_i$ is the minimal number of generators, and that the $n_i$ amalgamations are all along $\Z_p$. 
Thus we have $|G_i|\geq (p^3)^{n_i} p^{m_i}/p^{n_i} = p^{2n_i+m_i}$. 
On the other hand $p^r=|S|\geq |S/K_i|$, so comparing the two inequalities we conclude $2n_i+m_i\leq r$ as required.
\end{proof}

The overall picture of the proof of Theorems~\ref{thm:Heisenberg},~\ref{thm:specialGroups} is the following. Lemma~\ref{lem:Heisenberg},~\ref{lem:Suzuki} reduce the problem to find faithful actions of elements of $\Sp_{n,m}$ on a fixed manifold. Then Lemma~\ref{lem:main} reduces the problem from fixed manifolds to (potentially different) vector bundles having trivial Chern character whose rank and base space depends only on $(n,m)$.

In \S\ref{sec:buildingActionsOfCentralProducts} we introduce a general way to obtain faithful actions of iterated central products of groups. This reduces further the problem to find faithful actions only for the members of $\E_d$ and $\A_{m,d}$ from \eqref{eq:Edi} from Definition~\ref{def:EdAmd}.
In \S\ref{sec:constructingActions} we construct several actions of every member of $\Sp_{n,m}$ on various vector bundles over $\ctorus{n}$ using \S\ref{sec:buildingActionsOfCentralProducts} such that they all lift the same actions on $\ctorus{n}$.
In \S\ref{sec:main}, using number theoretic arguments, a suitable combination of Whitney sums and tensor products of the vector bundles constructed at \S\ref{sec:constructingActions} produces a faithful action on a vector bundle over $\torus{2n}$ that has trivial Chern character whose rank depending only on $(n,m)$. In light of the reduction steps described above, this establishes Theorem~\ref{thm:main} which is a common generalisation of Theorem~\ref{thm:Heisenberg},~\ref{thm:specialGroups}.

To review the necessary theory, the author found the following books useful: \cite{Suzuki1}, \cite{Suzuki2} for finite group theory, 
\cite{HatcherVB}, \cite{Husemoller} for vector bundles, 
\cite{park2008complex} for $K$-theory, and 
\cite{ComplexAbelianVarieties} for holomorphic bundles over complex torus.

\section{Induced central product action on external direct products}
\label{sec:buildingActionsOfCentralProducts}

\begin{summary}
In this section we introduce a general way to construct a faithful action of central product of groups on the external tensor product bundle from two faithful actions on arbitrary vector bundles. 
This enables building faithful actions of iterated central products recursively.
\end{summary}

\begin{defn}
Let $\rho:G\acts \pi$ be from Definition~\ref{def:bundleAction}. 
Define $B(\rho):=G/\ker(\baseAction{\rho})$ and write $\widehat{\rho}$ for the induced action of $B(\rho)$ on $X$, 
define a subgroup $D(\rho)$ of $G$ as the set of $g\in \ker(\baseAction{\rho})$ for which there exists a scalar $\lambda_\rho(g)$ such that $\rho(g)$ acts on every fibre of $\pi$ as multiplication by $\lambda_\rho(g)$ and  
denote the corresponding group morphism by $\lambda_\rho:D(\rho)\to \C^\times$.
Finally, if $\rho$ is faithful, then $D(\rho)\subseteq Z(G)$ and in this case, let $\bar Z(\rho):=Z(G)/D(\rho)$. 
To summarise, we have a commutative diagram of groups with exact first row. We shall refer to the non-dashed part as the \emph{diagram of the action} $\rho$.
\begin{equation}
\begin{tikzcd}[column sep=small]
1 \ar{rr} && \ker(\baseAction{\rho}) \ar{rr} && G \ar{rr} \ar["\baseAction{\rho}"]{dr} 
\ar["\rho"', dashed]{dl} 
&& B(\rho) \ar{rr}
\ar["\widehat{\rho}"']{dl}
&& 1\\
&D(\rho) \ar[dashed]{ur} \ar["\lambda_{\rho}", dashed]{r} & \C^\times \ar[dashed]{r} & \Aut_\cat^\pi(E) \ar["(-)^\pi", dashed]{rr} && \Aut_\cat(X) 
\end{tikzcd}
\label{eq:actionDiagram}
\end{equation}
\end{defn}

\begin{defn}\label{def:extBundles}
Let $\pi_i:E_i\to X_i$ be vector bundles for $i\in\{1,2\}$. 
Let $p_i:X_1\times X_2\to X_i$ be the projection map to the $i$th coordinate.
Define the \emph{external tensor product} of bundles $\pi_i$ by $E_1\extTensor E_2:=p_1^*(E_1)\otimes p_2^*(E_2)\xrightarrow{\pi_1\extTensor\pi_2} X_1\times X_2$. 
\end{defn}

\begin{lem}\label{lem:extActions}
Actions $\rho_i:G_i\acts \pi_i$ on vector bundles $\pi_i:E_i\to X_i$ for  $i\in\{1,2\}$ induce a natural action $\rho_1\extTensor \rho_2:G_1\times G_2\acts \pi_1\extTensor \pi_2$
for which $\baseAction{\rho_1\extTensor\rho_2} = \baseAction{\rho_1} \times \baseAction{\rho_2}$
and $\ker(\rho_1\extTensor \rho_2) = \{(g_1,g_2)\in D(\rho_1)\times D(\rho_2):\lambda_{\rho_1}(g_1)\lambda_{\rho_2}(g_2)=1\}$.
\end{lem}
\begin{proof}
The natural $G_1\times G_2$-action on $E_1\extTensor E_2$ using the linearity on the fibres is given by 
\begin{align*}
\rho_1\extTensor\rho_2:(g_1,g_2)&\mapsto (e_1\tensor e_2\mapsto \rho_1(g_1)(e_1) \tensor \rho_2(g_2)(e_2)) && 
\text{where } E_1\extTensor E_2=\coprod_{\mathclap{(x_1,x_2)\in X_1\times X_2}} \pi_1^{-1}(x_1)\tensor_\C \pi_2^{-1}(x_2) 
\end{align*}
while on $X_1\times X_2$ by $\baseAction{\rho_1\extTensor\rho_2}:=\baseAction{\rho_1}\times \baseAction{\rho_2}$.
It is straightforward to check that these definitions give rise to actions by vector bundle morphisms.

For the statement about the kernel, first note that $\ker(\rho_1\extTensor\rho_2)\subseteq 
\ker(\baseAction{\rho_1\extTensor\rho_2})=
\ker(\baseAction{\rho_1}\times\baseAction{\rho_2})=
\ker(\baseAction{\rho_1})\times \ker(\baseAction{\rho_2})$.
Now suppose $g=(g_1,g_2)\in \ker(\rho_1\extTensor \rho_2)$. 
Then by above $g_i\in\ker(\baseAction{\rho_i})$, and 
for every $e_1\tensor e_2\in E_1\extTensor E_2$, 
$e_1\tensor e_2=\rho_1(g_1)(e_1) \tensor \rho_2(g_2)(e_2)$.
This means from the identification of the tensor product, that there is $\mu(e_1,e_2)\in\C^\times$ such that $\rho_1(g_1)(e_1) = \mu(e_1,e_2) e_1$ and $\rho_2(g_2)(e_2) = \mu(e_1,e_2)^{-1} e_2$. 
Since the left-hand side of the first equation is independent of $e_2$, so must be $\mu(e_1,e_2)$. Similarly from the second equation we see that $\mu(e_1,e_2)$ is independent of $e_1$. 
Thus for some $\mu\in\C^\times$ for every for every $e_i\in E_i$, we have $\rho_1(g_1)(e_1) = \mu e_1$ and $\rho_2(g_2)(e_2) = \mu^{-1} e_2$. 
Thus by definition, $g_i\in D(\rho_i)$ and $\lambda_{\rho_1}(g_1)\lambda_{\rho_2}(g_2)=\mu \mu^{-1}=1$.
On the other hand, every such element belongs to $\ker(\rho_1\extTensor\rho_2)$.
\end{proof}

\begin{defn}\label{def:amalgamable}
Fix actions $\rho_i:G_i\acts \pi_i$ for $i\in\{1,2\}$. 
For central subgroups $D_i\subseteq Z(G_i)\cap D(\rho_i)$, call a group isomorphism $\phi:D_1\to D_2$ \emph{amalgamable for $\rho_1$ and $\rho_2$} if $\lambda_{\rho_1} = \lambda_{\rho_2}\circ \phi$ on $D_1$. The set $A(\rho_1,\rho_2)$ of amalgamable maps form a poset where $\phi_1\leq \phi_2$ if and only if $\phi_1$ is a restriction of $\phi_2$.
\end{defn}

\begin{lem}\label{lem:makePhiAmalgamable}
Let $G_1\cprod[\phi] G_2$ be given with amalgamation $\phi:D_1\to D_2$ and assume that $G_1$ is finite.
Let $\rho_1:G_1\acts \pi_1$ be faithful such that $D_1\subseteq D(\rho_1)$. 
Then there exists $z\in Z(G_2)$ of order $r:=|D_1|$ such that 
for every action $\rho_2:G_2\acts \pi_2$ satisfying $z\in D(\rho_2)$ with $\lambda_{\rho_2}(z)=\mu_r$, 
we have $\phi\in A(\rho_1,\rho_2)$. 
If addition, $\phi$ is a maximal central isomorphism (c.f. Definition~\ref{def:cprod}), then it is maximal in the resulting $A(\rho_1,\rho_2)$. 
\end{lem}
\begin{proof}
$G_1$ is finite, so the image $\lambda_{\rho_1}(D(\rho_1))\subset \C^\times$ is also finite, hence is a cyclic group. Then from the faithfulness of $\rho_1$ we see that $D(\rho_1)$ is also cyclic, hence so is $D_1$.
Let $z_1$ be a generator of $D_1$, let $z:=\phi(z_1)\in Z(G_2)$ which is a generator of $D_2\subseteq Z(G_2)$, because $\phi$ is an isomorphism. 
By the assumptions on $\rho_1$, we my assume $\lambda_{\rho_1}(z_1)=\mu_r$. 
On the other hand, if $\rho_2:G_2\acts \pi_2$ satisfies $z\in D(\rho_2)$ with $\lambda_{\rho_2}(z)=\mu_r$, then by construction $\lambda_{\rho_1}(z_1)=\mu_r=\lambda_{\rho_2}(\phi(z_1))$, so $\lambda_{\rho_1}=\lambda_{\rho_2} \circ \phi$. 
Finally, noting that $z\in D(\rho_2)$, we have $D_2\subseteq D(\rho_2)$, so $D_i\subseteq Z(G_i)\cap D(\rho_i)$, thus $\phi\in A(\rho_1,\rho_2)$.

For the second part, $\phi\leq \bar{\phi}$ in $A(\rho_1,\rho_2)$, then $\bar{\phi}$ is an isomorphism between central subgroups of $G_1$ and $G_2$. Thus if $\phi$ is a maximal central isomorphism, then $\phi=\bar\phi$, showing that $\phi$ is maximal in $A(\rho_1,\rho_2)$.
\end{proof}

\begin{prop}\label{prop:centralProductAction}
Let $\rho_i:G_i\acts\pi_i$ where $\pi_i:E_i\to X_i$ for $i\in\{1,2\}$. 
Then every $\phi\in A(\rho_1,\rho_2)$ induces a natural action  $\rho_1\cprod[\phi]\rho_2:G_1\cprod[\phi]G_2\acts \pi_1\extTensor\pi_2$
with diagram 
\begin{equation}
\begin{tikzcd}[column sep=small, row sep=small]
1 \ar{rr} && \ker(\baseAction{\rho_1\cprod[\phi] \rho_2}) \ar{rr} && G_1\cprod[\phi] G_2 \ar{rr} \ar["\baseAction{\rho_1\cprod[\phi] \rho_2}"']{dr} 
&& B(\rho_1)\times B(\rho_2) \ar{rr}
\ar["\widehat{\rho_1}\times \widehat{\rho_2}"]{dl}
&& 1\\
& & & & & \Aut_\cat(X_1\times X_2) 
\end{tikzcd}.
\label{eq:actionDiagramOfCProd}
\end{equation}
If in addition both $\rho_1$ and $\rho_2$ are faithful, then the poset 
$A(\rho_1,\rho_2)$ has a unique greatest element $\bar\phi$ for which 
$\rho_1\cprod\rho_2:=\rho_1 \cprod[\bar\phi]\rho_2$ is faithful with $\bar Z(\rho_1\cprod\rho_2)\isom \bar Z(\rho_1)\times \bar Z(\rho_2)$.
\end{prop}
\begin{proof}
Let $\phi\in A(\rho_1,\rho_2)$. 
By Definitions~\ref{def:cprod},~\ref{def:amalgamable} and Lemma~\ref{lem:extActions}, we have $K_\phi\subseteq \ker(\rho_1\extTensor \rho_2)$, 
hence $\rho_1\extTensor \rho_2$ indeed induces an action of $G_1\times G_2/K_\phi=G_1\cprod[\phi] G_2$ on $\pi_1\extTensor\pi_2$ which we denote by $\rho:=\rho_1\cprod[\phi] \rho_2$. 
By definitions, the diagrams \eqref{eq:actionDiagram} of these actions gives rise to the following commutative diagram \eqref{eq:cprodDiagram} of groups with all rows and the leftmost two columns being short exact. (We omit displaying the first and last trivial groups of the short exact sequences.) Then by the nine lemma, the commutative diagram can be completed with a short exact right column. Hence $B(\rho_1)\times B(\rho_2) \isom B(\rho)$ is compatible with the action as stated.
\begin{equation}\label{eq:cprodDiagram}
\begin{tikzcd}[column sep=small,row sep=small]
K_\phi \ar[equal]{r} \ar{d} &
K_\phi \ar{rr} \ar{d} && 
1 \ar[dashed]{d}
\\
\ker(\baseAction{\rho_1})\times\ker(\baseAction{\rho_2}) \ar{r} \ar{dd} &
G_1\times G_2 \ar{rr} \ar["\baseAction{\rho_1}\times \baseAction{\rho_2}"']{dr} \ar{dd} && 
B(\rho_1)\times B(\rho_2) \ar["\widehat{\rho_1}\times \widehat{\rho_2}"]{dl} \ar[dashed,"\simeq" {anchor=south, rotate=90}, "\eta_\phi"]{dd}
\\
 & & \Aut_\cat(X_1\times X_2)
\\
\ker(\baseAction{\rho}) \ar{r} & 
G_1\cprod[\phi] G_2 \ar{rr} \ar["\baseAction{\rho}"]{ur} &&
B(\rho) \ar["\widehat{\rho}"']{ul}
\end{tikzcd}
\end{equation}

For the second part, note that the projections $p_i:\ker(\rho_1\extTensor \rho_2)\to D(\rho_i)$ are bijections by Lemma~\ref{lem:extActions} and the faithfulness of $\rho_i$. 
Thus $\bar{\phi}:=p_2 \circ p_1^{-1} :\im(p_1)\to\im(p_2)$ is in $A(\rho_1,\rho_2)$ by Definition~\ref{def:amalgamable}. On the other hand, $K_{\bar{\phi}}=\ker(\rho_1\extTensor\rho_2)$ so is indeed maximal in $A(\rho_1,\rho_2)$ and has kernel $\ker(\rho_1\cprod \rho_2)=\ker(\rho_1\extTensor\rho_2)/K_{\bar{\phi}}=1$, so $\rho_1\cprod \rho_2$ is indeed faithful.
Similarly to Lemma~\ref{lem:extActions}, for any $\phi\in A(\rho_1,\rho_2)$, we have $D(\rho_1\cprod[\phi] \rho_2)=D(\rho_1)\times D(\rho_2)/K_\phi$ and on the other hand $Z(G\cprod[\phi] G)=Z(G_1)\times Z(G_2)/K_\phi$. Hence using the nine lemma as above (or the third isomorphism theorem) for $\phi=\bar{\phi}$ gives $\bar Z(\rho_1\cprod\rho_2)\isom \bar Z(\rho_1)\times \bar Z(\rho_2)$.
\end{proof}

\begin{rem}
Lemma~\ref{lem:makePhiAmalgamable} together with  Proposition~\ref{prop:centralProductAction} essentially state that given $\rho_1$, we can choose $\rho_2$ suitably so that the resulting action is of the predefined central product of these two groups and is also faithful.
\end{rem}

\section{Constructing actions on holomorphic bundles over torus}
\label{sec:constructingActions}
\begin{summary}
For every $(n,m)\in\Np^2$ and every $G\in \Sp_{n,m}$ we construct in two ways several actions of $G$ on bundles over $\ctorus{n}$ lifting the exact same action on $\ctorus{n}$. These bundles depend on $G$ but their rank depends only on $(n,m)$. 
The first construction uses \S\ref{sec:buildingActionsOfCentralProducts} and recursion to yield a faithful action and a (typically non-faithful) action on a line bundle (Proposition~\ref{prop:actionOnG}). 
The second construction produces a (typically not faithful) $G$-action on any pullback bundle along $\ctorus{n}\to \ctorus{n}/\Z_d^{2n}\isom \ctorus{n}$ using a suitable surjection $G\to \Z_d^{2n}$ (Proposition~\ref{prop:actionChernClass}). This enables to control the Chern character well at \S\ref{sec:main}.
\end{summary}

In this section we work in the category of complex manifolds and holomorphic vector bundles. Actually for the big picture, smooth vector bundles would be sufficient, but the former one is a bit more natural language to use here. 
For the constructions, we use the elementary observation that 
for any $z\in \Z_m$ of order, say $r$, there is a generator $g$ of $\Z_m$ such that $\tfrac{m}{r}g=z$. This follows from the fact the natural map $\Aut(\Z_m,+)\isom (\Z_m^\times, \cdot) \to (\Z_r^\times, \cdot) \isom \Aut(\Z_r,+)$ is surjective.

\begin{lem}\label{lem:actionOfEdi}
There is holomorphic line bundle $\xi$ over $\ctorus{1}$ such that 
for every $d\in\Np$ and every $E\in\E_d$ from \eqref{eq:Edi} and $z\in Z(E)$ of order, say $r$, 
there is a faithful action $\rho:E\acts\xi_d$ where $\xi_d:=\xi^{\tensor d}$ 
with diagram from \eqref{eq:actionDiagram}:
\begin{equation}
\begin{tikzcd}[column sep=small, row sep=small]
1 \ar{rr} && \ker(\baseAction{\rho}) \ar{rr} && E \ar{rr} \ar["\baseAction{\rho}"']{dr} 
&& \Lambda_d \ar{rr}
\ar["\nu_d=\widehat{\rho}"]{dl}
&& 1\\
& & & & & \Biholo(\ctorus{1}) 
\end{tikzcd}
\label{eq:actionDiagramOfE}
\end{equation}
such that $\bar Z(\rho)=1$ and 
$z\in D(\rho)$ with $\lambda_{\rho}(z)=\mu_r$,
where we define $\Lambda_d:=(\tfrac{1}{d}\Lambda)/\Lambda \isom \Z_d^2$ for $\Lambda=\Z+i\Z \subset \C$, and $\nu_d:l\mapsto(t \mapsto l+t)$.
\end{lem}

\begin{proof}

First we construct $\xi_d$ as a quotient of $\C^2$ as given by Appell--Humbert theorem described at \cite[\S2.2]{ComplexAbelianVarieties} in great details. Let $\chi:\C\to\C^\times,l\mapsto \exp(\pi i \Re(l)\Im(l))$, where $\Re$ and $\Im$ are the real and imaginary parts. 
Let $\langle v, l\rangle:=v\bar{l}$ be the usual Hermitian form on $\C$ and 
for every $l\in \C$ define the following holomorphic maps: \[\phi(l):\C^2\to \C^2, (v,t)\mapsto (v+l,\chi(l)^d \exp(\pi \langle v, l\rangle + \tfrac{\pi}{2} \langle l, l\rangle)^d t)\] and 
$r_2(l):\C^2\to \C^2,(v,t)\mapsto (v,\exp(2\pi i l)t)$. 
Short calculation shows that $\phi(l_1 + l_2) = \phi(l_1)\circ \phi(l_2)\circ r_2(d \Im(l_1) \Re(l_2))$ and hence 
$\phi(l_1)\circ \phi(l_2) = \phi(l_2)\circ \phi(l_1)\circ r_2(d \Im\langle l_2, l_1\rangle)$
for any $(l_1,l_2)\in \C^2$.
In particular we see that for additive groups  $H\in\{\Lambda,\tfrac{1}{d}\Z+i\Z,\Z+\tfrac{i}{d}\Z\}$ we have group actions 
$\phi|_H:H\to \Biholo(\C^2)$. 
Let $L_d:=\C^2/(\phi|_{\Lambda})$ be the orbit space.  The projection to the first coordinate induces a holomorphic line bundle $\xi_d:L_d\to\ctorus{1}$. 
Note that $\xi_d=\xi_1^{\tensor d}$, so we may put $\xi:=\xi_1$.

Next, we give a suitable action on $\xi_d$. 
Every $(\phi|_{\Lambda})$-equivariant holomorphic map $f:\C^2\to \C^2$ induce natural holomorphic maps $[f]:L_d\to L_d$ and $\overline{[f]}:\ctorus{1}\to\ctorus{1}$ such that $\overline{[f]}\circ \xi_d = \xi_d\circ [f]$. 
Note that $z\in Z(E)=\langle \gamma\rangle \isom \Z_d$, so by the observation, we may pick a generator $\bar{\gamma}$ of $Z(E)$ such that $\bar\gamma^{d/r} = z$. Let $1\leq k<d$ with $\gcd(d,k)=1$ be given by $\bar{\gamma}^k=\gamma$ and let $j$ is given by $E=E(d,j)$. 
Define an action $\rho$ on $\xi_d$ by specifying the following biholomorphic maps on the generators from \eqref{eq:Edi}:
\begin{align*}
\rho(\alpha)&:=[\phi(\tfrac{k}{d})\circ r_2(\tfrac{jk}{d^2})],&
\rho(\beta)&:=[\phi(\tfrac{i}{d})\circ r_2(\tfrac{jk}{d^2})],& 
\rho(\gamma)&:=[r_2(\tfrac{k}{d})],\\
\bar\rho(\alpha)&:=\overline{[\phi(\tfrac{k}{d})]},&
\bar\rho(\beta)&:=\overline{[\phi(\tfrac{i}{d})]},&
\bar\rho(\gamma)&:=\overline{[\phi(0)]}.
\end{align*}
By above, we have 
$\rho(\alpha)^d = [\phi(k)\circ r_2(\tfrac{jk}{d})] = \phi(\gamma)^j =
[\phi(i)\circ r_2(\tfrac{jk}{d})] = \rho(\beta)^d$, 
$[\rho(\alpha), \rho(\beta)] = [r_2(d \Im\langle \tfrac{i}{d}, \tfrac{k}{d}\rangle)]=\rho(\gamma)$ and the rest of the defining relations of $E$ at \eqref{eq:Edi} are also satisfied similarly.
To see faithfulness, suppose an arbitrary $\alpha^a\beta^b\gamma^c\in E$ acts trivially. Then we may assume that $0\leq a,b,c<d$ and since $\rho(\alpha^a\beta^b\gamma^c) = [\phi(\tfrac{ak}{d})\circ \phi(\tfrac{bi}{d})\circ r_2(k\tfrac{j(a+b)+dc}{d^2})]$, we must have $ak/d, bi/d,kj(a+b)/d^2+kc/d\in\Z$, but this is only possible if $a=b=c=0$ because $\gcd(d,k)=1$, hence the action is indeed faithful.

Finally we check the stated properties. By definition $D(\rho)=\langle \gamma\rangle=Z(E)$, so $\bar Z(\rho)=1$. 
We have an isomorphism $B(\rho)\to \Lambda_d, g\ker(\baseAction{\rho})\mapsto l$ given by $\baseAction{\rho}(g)=\overline{[\phi(l)]}$.
By construction, $\bar{\gamma}$ acts as multiplication by $\mu_d$ on all fibres, so $z=\bar{\gamma}^{d/r}$ acts as multiplication by $\mu_d^{d/r}=\mu_r$ on all fibres as stated.
\end{proof}
\begin{rem}
The construction is built on the natural action of the Heisenberg group from \cite[\S6]{ComplexAbelianVarieties} (which also appeared at \cite{zarhin2014}, \cite{CsPSz} and at \cite{Riera} in a smooth but not holomorphic version). Since this covers only $E(d,0)\in\E_d$, we twisted the centre to include every other member of $\E_d$ that have exponent bigger than $d$.
\end{rem}

\begin{lem}\label{lem:actionOfA}
Denote by $\theta:\C\to\{0\}$ the trivial holomorphic line bundle over the $1$ point space.
Let $A\in\A_{m,d}$ and let $z\in A$ be arbitrary of order $d$. 
Then there is a faithful action $\rho_{A,z}:A\acts \theta^{\oplus m}$ and a (typically not faithful) action $\rho_{A,z}':A\acts \theta$
such that for both $\rho\in\{\rho_{A,z},\rho_{A,z}'\}$, we have $z\in D(\rho)$ with $\lambda_{\rho}(z)=\mu_d$ and identical corresponding diagram from \eqref{eq:actionDiagram}:
\begin{equation}
\begin{tikzcd}[column sep=small, row sep=small]
1 \ar{rr} && \ker(\baseAction{\rho}) \ar{rr} && A \ar{rr} \ar["\baseAction{\rho}"']{dr} 
&& 1 \ar{rr}
\ar["\widehat{\rho}"]{dl}
&& 1\\
& & & & & \Biholo(\{0\})
\end{tikzcd}.
\label{eq:actionDiagramOfA}
\end{equation}
\end{lem}

\begin{proof}
Without loss of generality, let $A=\prod_{i=1}^m \Z_{n_i}$. Let $z=(z_i)_{i=1}^m$ and $d_i$ be the order of $z_i$ in $\Z_{n_i}$. 
Using the Chinese remainder theorem to rearrange the elementary divisors, we may assume that $d_i\mid d_j$ for $i\leq j$. 
Let $g_i$ be the generator of $\Z_{n_i}$ provided by the observation applied to $z_i\in\Z_{n_i}$, hence $\tfrac{n_i}{d_i} g_i=z_i$.
Let $e_i\in A$ be the element whose component are all $0$'s except the $i$th one where it is $g_i$. 

We define an action $\rho_{A,z}$ of $A$ on $\C^m$ by giving it on the generating set $\{e_1,\dots,e_m\}$.
Let $e_i\in A$ act on the $j$th coordinate of $\C^m$ as multiplication 
by $\mu_{n_i}$ if $j=i$; 
by $\mu_{n_m}^{1-d_m/d_j}$ if $j\neq i=m$; 
and trivially otherwise. 
To see that this action is faithful, assume that $a:=\sum_{i=1}^m a_i e_i$ acts trivially for some $0\leq a_i < n_i$. Then by definition the action on the $m$th coordinate of $\C^m$ is multiplication by $\mu_{n_m}^{a_m}=1$, so $a_m=0$. Hence the action on the $j$th coordinate for $1\leq j < m$ is multiplication by $\mu_{n_j}^{a_j}=1$, so $a_j=0$, thus $a=0$, and faithfulness follows. 
Finally $z\in D(\rho)$ with $\lambda_{\rho}(z)=\mu_d$ is clear, because by construction $z=\sum_{i=1}^m \frac{n_i}{d_i}e_i$ acts on the $j\neq m$th coordinate by multiplication by 
$\mu_{n_j}^{n_j/d_j} \mu_{n_m}^{(1-d_m/d_j)n_m/d_m} = 
\mu_{d_m}=
\mu_d$ and on the $m$th coordinate by $\mu_{n_m}^{n_m/d_m}=\mu_{d_m}=\mu_d$.

To construct the other action $\rho_{A,z}'$ on $\C$, we may take the invariant space of the $m$th factor of $\C^m$ of $\rho_{A,z}$. More explicitly, let $e_j$ act trivially for $1\leq j<m$ and let $e_m$ act as multiplication by $\mu_{n_m}$.
\end{proof}

In the rest of the paper, we use the following notation.
\begin{defn}\label{def:Gamma}
For $n\in\Np$ and $1\leq k\leq n$, denote by 
$p_k:\ctorus{n}\to \ctorus{1}$  the natural projection to the $k$th factor and define $\omega_k:=p_k^*(\omega)$ where $\omega\in H^2(\ctorus{1},\Z)\isom \Z$ is a fixed generator.
Let $\Gamma_n^\bullet\subseteq H^\bullet(\ctorus{n},\Z)$ be the subring generated by $\{1,\omega_1,\dots,\omega_n\}$ where $1\in H^\bullet(\ctorus{n},\Z)$ is the multiplicative identity and 
let $\Gamma_n^r:=\Gamma_n^\bullet\cap H^{2r}(\ctorus{n},\Z)$. 
\end{defn}

\begin{prop}\label{prop:actionOnG}
For every $(n,m)\in\Np^2$, $G\in\Sp_{n,m}$ with $d\in\Np$ from \eqref{eq:typeS}, there exists a holomorphic line bundle 
$\pi_\mathrm{l}$ over $\ctorus{n}$ such that 
the first Chern class $c_1(\pi_\mathrm{l})\in \Gamma_n^1$, 
and there is an action $\rho_\mathrm{l}:G\acts \pi_\mathrm{l}$, 
and a faithful action $\rho_\mathrm{f}:G\acts \pi_\mathrm{f}$ where $\pi_\mathrm{f}:= \pi_\mathrm{l}^{\oplus m}$ such that  
both actions have the same diagram from \eqref{eq:actionDiagram}:
\begin{equation}
\begin{tikzcd}[column sep=small, row sep=small]
1 \ar{rr} && \ker(\baseAction{\rho}) \ar{rr} && G \ar["\eta_G"]{rr} \ar["\baseAction{\rho_G}:=\baseAction{\rho}"']{dr} 
&& \Lambda_d^n \ar{rr}
\ar["{\nu:=\underbrace{\scriptstyle\nu_d\times \dots \times \nu_d}_n}"]{dl}
&& 1\\
& & & & & \Biholo(\ctorus{n}) 
\end{tikzcd}
\label{eq:actionDiagramOfG}
\end{equation}
where $\rho\in\{\rho_\mathrm{f}, \rho_\mathrm{l}\}$ and $\nu_d:\Lambda_d \to\Biholo(\ctorus{1})$ is given at \eqref{eq:actionDiagramOfE}.
\end{prop}

\begin{proof}
Let $\xi_d$ and $\theta$ be given by Lemma~\ref{lem:actionOfEdi},~ \ref{lem:actionOfA}. We claim that $\pi_\mathrm{l}:=\xi_d^{\extTensor n}\extTensor \theta$ and $\pi_\mathrm{f}=\rho_\mathrm{l}^{\oplus m}\isom\xi_d^{\extTensor n}\extTensor (\theta^{\oplus m})$ satisfy the statement. 
Note that $c_1(\pi_l)=c_1(\xi_d^{\extTensor n})=c_1(\bigotimes_{k=1}^n p_k^*(\xi_d))=\sum_{k=1}^n p_k^*(c_1(\xi_d))=\delta\sum_{k=1}^n \omega_k\in \Gamma_n^1$ where $c_1(\xi_d)=\delta \omega$. (Actually $\delta=\pm d$, but this is not relevant.)

For $1\leq k\leq n$, let $H_k$ be the group obtained by taking only the first $k$ groups at \eqref{eq:typeS}. 
We claim that for $1\leq k\leq n$, there is a faithful action $\rho_k:H_k\acts \xi_d^{\extTensor n}$ such that $\bar Z(\rho_k)=1$. 
The case $k=1$ is given by Lemma~\ref{lem:actionOfEdi}. 
Else by induction there is a faithful action $\rho_k:H_k\acts \xi_d^{\extTensor k}$ with $\bar Z(\rho_k)=1$. 
The latter condition implies that the domain of $\phi_k$ lies in $D(\rho_k)$, so Lemma~\ref{lem:makePhiAmalgamable} is applicable for $H_{k+1}=H_k\cprod[\phi_k] E_{k+1}$ and this together with Lemma~\ref{lem:actionOfEdi} gives a faithful action $\rho:E_{k+1}\acts \xi_d$ with $\bar Z (\rho)=1$  such that $\phi_k$ is a maximal element of $A(\rho_k,\rho)$. 
Then Proposition~\ref{prop:centralProductAction} produces a faithful action $\rho_{k+1}=\rho_k\cprod \rho:H_{k+1}\acts \xi_d^{\extTensor k+1}$ for which $\bar Z(\rho_{k+1})\isom \bar Z (\rho_k)\oplus \bar Z (\rho)\isom 1$, so the claim is proved.

Finally, for $G=H_n\cprod[\phi_n] A$, we define the required actions in a similar fashion. Lemma~\ref{lem:actionOfA} produces a faithful action $\rho:A\acts \theta^{\oplus m}$ and an action $\rho':A\acts \theta$ such that 
$\rho_\mathrm{f}:=\rho_n\cprod \rho:G\acts \pi_\mathrm{f}$ is faithful and $\rho_\mathrm{l}:=\rho_n\cprod \rho:G\acts \pi_\mathrm{l}$ is an action by Lemma~\ref{lem:makePhiAmalgamable} and Proposition~\ref{prop:centralProductAction}. 
The diagrams \eqref{eq:actionDiagramOfG} corresponding to $\rho_\mathrm{f}$ and $\rho_\mathrm{l}$ are obtained from the diagrams \eqref{eq:actionDiagramOfE} and \eqref{eq:actionDiagramOfA} using \eqref{eq:actionDiagramOfCProd} following the above construction. 
\end{proof}

The second construction produces a (typically non-faithful) $G$-action on a vector bundle with predefined Chern character.

\begin{lem}\label{lem:polynomial}
Let $n\in\Np$, 
$P_n$ denote the set of non-empty subsets of $\{1,\dots,n\}$. 
Consider the polynomial ring $R[x_1,\dots,x_n]$ over a commutative ring $R$. 
Then for every $f:P_n\to R$, there exist $g_J:J\to R$ for every $J\in P_n$ such that 
\begin{equation}
\sum_{J\in P_n} \prod_{j\in J} (1+g_J(j)x_j) = 
2^n-1+\sum_{I\in P_n} f(I) \prod_{i\in I} x_i.
\label{eq:polynomial}
\end{equation}
\end{lem}
\begin{proof}
Pick $f:P_n\to R$. 
For $J\in P_n$, let $g_J:J\to R$ defined by $g_J(\max(J)):=h(J):=1-2^{n-\max(J)} +\sum_{K\geq J} (-1)^{|K\setminus J|} f(K)$ and $g_J(j):=1$ for $j< \max(J)$, where 
we make $P_n$ a poset by defining $I\leq J$ if $I\subseteq J$ and $\max(I)=\max(J)$. 
We claim that these $g_J$'s satisfy the statement. 
Indeed, the constant terms of \eqref{eq:polynomial} match. Pick $I\in P_n$. 
By expanding the brackets, the coefficient of $\prod_{i\in I} x_i$ at the left-hand side of \eqref{eq:polynomial} equals 
\[
\sum_{J\supseteq I}\prod_{i\in I} g_J(i) = 
\sum_{\mathclap{J\supseteq I,J\ngeq I}} 1 +  
\sum_{J\geq I} h(J) =
\sum_{J\geq I} \sum_{K\geq J} (-1)^{|K\setminus J|} f(K)
= f(I)
\]
as required. Here the first equality follows from the definition of $g_J$, 
the second one from $|\{J\in P_n:J\supseteq I,J\ngeq I\}|=2^{n-|I|}-2^{\max(I)-|I|}$, 
$|\{J\in P_n: J\geq I\}|=2^{\max(I)-|I|}$, the formula of $h(J)$ and from the fact that $\max(J)=\max(I)$ for $J\geq I$, 
while the last equation is the Exclusion-Inclusion principle (or more formally the dual form of the Möbius inversion formula for $P_n$ where the Möbius function is $\mu(J,K)=(-1)^{|K\setminus J|}$).
\end{proof}

\begin{lem}\label{lem:integralCh}
For any $n\in\Np$ and for any $\gamma\in\Gamma_n^\bullet$ (c.f. Definition~\ref{def:Gamma}), there is a holomorphic vector bundle $\pi_\gamma$ of rank $n$ such that $\ch(\pi_\gamma)-\gamma\in H^0(\ctorus{n},\Q)$.
\end{lem}
\begin{proof}
Recall Definition~\ref{def:Gamma} and let $P_n$ as at Lemma~\ref{lem:polynomial}.
Using that $\omega_k\cuptimes\omega_k=p_k^*(\omega\cuptimes\omega)=p_k^*(0)=0$ for $1\leq k\leq n$, we may write $\gamma=\gamma_0+2^n-1+\sum_{I\in P_n}f(I)\prod_{i\in I}\omega_i$ for some $\gamma_0\in H^0(\ctorus{n},\Q)$ and $f:P_n\to \Z$. 
Since $\ctorus{1}$ is a compact Riemann surface, there is a holomorphic line bundle $\psi$ over $\ctorus{1}$ such that $c_1(\psi)=\omega$. (Actually $\psi$ could be taken to be $\xi^{\tensor\pm 1}$ from Lemma~\ref{lem:actionOfEdi}, but this is not relevant.) Let $\psi_k:=p_k^*(\psi)$, a holomorphic line bundle over $\ctorus{n}$. 
Let $g_J:P_n\to \Z$ by given by Lemma~\ref{lem:polynomial} for the $f$ above, and 
define a holomorphic vector bundle over $\ctorus{n}$ by 
$\pi_\gamma:=\bigoplus_{J\in P_n} \bigotimes_{j\in J} \psi_j^{\tensor g_J(j)}$.
Note that $\ch(\psi_j^{\tensor a})=\exp(c_1(p_k^* (\psi_0)))^a=\exp(a \omega_k)=1+a\omega_k$, so using that $\ch$ is a ring morphism, by construction we conclude that 
$\ch(\pi_\gamma) = \gamma-\gamma_0$, hence $\ch(\pi_\gamma)-\gamma = \gamma_0\in H^0(\ctorus{n},\Q)$ as required. 
Finally note that the rank of $\pi_\gamma$ is $2^n-1$ (independently of $\gamma$), but applying \cite[Part II, \S9.1, Theorem~1.5]{Husemoller}, we may assume that the rank of $\pi_\gamma$ is exactly $n$. 
\end{proof}
\begin{rem}
If we restrict our attention to smooth vector bundles over $\torus{n}$, then  $\Gamma_n^\bullet$ can be replaced by $H^{2\bullet}(\torus{n},\Z)$ in the statement. In short, $H^{2\bullet}(\torus{n},\Z)\subseteq \ch(K^0(\torus{n}))$. 
A proof of this uses a similar construction from line bundles with the key observations being that $H^2(\torus{n})$ classifies sooth real line bundles and that $H^{2r}(\torus{n},\Z)\isom\bigwedge^r H^2(\torus{n},\Z)$. Note that in this case the application of Lemma~\ref{lem:polynomial} can be replaced by basic results from $K$-theory.
\end{rem}

\begin{prop}\label{prop:actionChernClass}
For every $(n,m)\in\Np^2$,  $G\in\Sp_{n,m}$ with $d\in\Np$  from \eqref{eq:typeS} and every 
$\chi\in\bigoplus_{k=0}^n d^{2k}\Gamma_n^k$, 
there exists a holomorphic vector bundle $\pi_\chi$ over $\ctorus{n}$ of rank $n$ 
with $\ch(\pi_\chi) -\chi \in H^0(X,\Q)$ 
having a (typically not faithful) action $\rho_\chi:G\acts \pi_\chi$ such that $\baseAction{\rho_\chi} = \baseAction{\rho_G}$ from \eqref{eq:actionDiagramOfG}.
\end{prop}

\begin{proof}
From the diagram \eqref{eq:actionDiagramOfG}, we have $\ctorus{n}/\baseAction{\rho_G}=\ctorus{n}/\nu$, and since $\nu:\Lambda_d^n\to \Biholo(\ctorus{n}),l\mapsto(t \mapsto l+t)$, we further have $\ctorus{n}/\nu \isom \ctorus{n}$. 
Let $q:\ctorus{n}\to \ctorus{n}/\baseAction{\rho_G}$ be the projection to the orbit space of $\baseAction{\rho_G}$. 
We construct the required bundle together with a $G$-action in the form of the following pullback where the vector bundles are specified below. 
\[\begin{tikzcd}
q^*(E) \dar["\pi_\chi"] \rar & E \rar[equal] \dar["\pi"] & E \dar["\pi_\gamma"] \\
\ctorus{n} \rar["q"] \arrow[rr, bend right, "\epsilon_d"]
 & \ctorus{n}/\baseAction{\rho_G} \rar["\simeq"] & \ctorus{n}.
\end{tikzcd}\]

Note that by above, $\epsilon_d$ is given by multiplication by $d$, so $\epsilon_d^*(\omega)=d^2\omega$ (c.f. Definition~\ref{def:Gamma}), thus the induced map $\epsilon_d^*:\Gamma_n^k\to \Gamma_n^k$ is multiplication by $d^{2k}$. 
Hence we may pick $\gamma\in \bigoplus_{r=0}^n \Gamma_n^r$ such that $\epsilon_d^*(\gamma)=\chi$. 
Now let $\pi_\gamma$ be the vector bundle over $\ctorus{n}$ of rank $n$ produced by Lemma~\ref{lem:integralCh} and let $\pi_\chi:=\epsilon_d^*(\pi_\gamma)$. 
By construction, $\ch(\pi_\chi) - \chi= \epsilon_d^*(\ch(\pi_\gamma))-\epsilon_d^*(\gamma) \in \epsilon_d^*(H^0(\ctorus{n},\Q))=H^0(\ctorus{n},\Q)$ as required.
The natural action induced by the pullback structure along $q$ induces the desired action $\rho_\chi:G\acts\pi_\chi$. 
In more details, by definition $q^{*}(E)=\{(t,e)\in \ctorus{n}\times E:q(t)=\pi(e)\}$ and since $q(\baseAction{\rho_G}(g)(t))=q(t)$ for every $t\in\ctorus{n}$ and $g\in G$, we may define 
$\rho_\chi: G\to \Biholo(q^*(E)), g\mapsto \big((t,e)\mapsto (\baseAction{\rho_G}(g)(t),e)\big)$ and 
$\baseAction{\rho_\chi}:=\baseAction{\rho_G}$.
\end{proof}

\section{Proof of theorems}
\label{sec:main}
\begin{summary}
For every $(n,m)\in\Np^2$ and every $G\in \Sp_{n,m}$, using number theoretic arguments (Lemma~\ref{lem:numberTheory}) obtained from the modulo Waring problem, we combine the vector bundles of \S\ref{sec:constructingActions} using Whitney sums and tensor powers to produce a faithful $G$-action on a vector bundle over $\torus{2n}$ with trivial Chern character whose rank depend only on $(n,m)$ (while a priori the constructed bundle itself depends on $G$). 
Then Lemma~\ref{lem:main} produces a fixed compact manifold on which every member of $\Sp_{n,m}$ act faithfully and the direct product of such manifolds has the property stated at Theorem~\ref{thm:Heisenberg}/\ref{thm:specialGroups}. 
At the end, we investigate the simplest special case of the construction of this paper to obtain a very explicit counterexample to Ghys' conjecture (Remark~\ref{rem:ghys}).
\end{summary}

\begin{lem}\label{lem:numberTheory}
For every $(n,m)\in \Np^2$ 
there is $N(n,m)\in \Np$ such that 
for every $(\delta_1,\dots,\delta_n)\in(\Z\setminus\{0\})^n$  
there is a multiset $A$ of integers of cardinality $N(n,m)$ containing $1$ with multiplicity at least $m$ such that $\sum_{a\in A}a^k\in \delta_k\Z$ for every $1\leq k\leq n$.
\end{lem}
\begin{proof}
For every $k\in\Np$, denote by $M_{k}$ the smallest positive integer $M$ with the property that modulo any natural number, $-1$ can be expressed as a sum of $M$ (possibly involving $0$'s) $k$th powers. The existence of this follows from the modulo Waring problem and we have $M_{k}\leq 4k$, see \cite[p.~186, Theorem~12]{Hardy1922}. 

Pick $(n,m)\in \Np^2$. We show that $N(n,m):=(m+1)\prod_{k=2}^n (M_{k}+1)$ satisfies the statement. 
For a multiset $A$, let $p_k(A):=\sum_{a\in A}a^k$. 
Let $A_1:=\{1^{(m)},-m^{(1)}\}$, the multiset containing $1$ with multiplicity $m$ and $-m$ with multiplicity $1$. 
For any $2\leq k\leq n$, pick $B_k$ of cardinality $M_k$ such that $p_k(B_k)\equiv -1\pmod {|\delta_k|}$ which is possible from the definition of $M_k$. 
Let $A_k:=\{1^{(1)}\}+B_k$ the multiset obtained from $B_k$ by increasing the multiplicity of $1$ by one. 
By construction, $\delta_k\mid p_k(A_k)$ for all $1\leq k\leq n$.
Let $A:=A_1\cdot {\dots} \cdot A_n:=\{\prod_{k=1}^n a_k: 1\leq k\leq n,a_k\in A_k\}$, a multiset of cardinality $M(n,m)$ containing $1$ with multiplicity at least $m$. 
Finally note that $\delta_k\mid p_k(A_k)\mid\prod_{i=1}^n p_k(A_i)=p_k(A)$ by construction for any $1\leq k\leq n$, so $A$ is as required.
\end{proof}

\begin{prop}\label{prop:actonOnTrivialChernClass}
For every $(n,m)\in\Np^2$, there is an integer $R(n,m)\geq n$ such that for every 
$G\in\Sp_{n,m}$ from \eqref{eq:typeS}, 
there is a holomorphic vector bundle $\pi_G$ over $\ctorus{n}$ of rank $R(n,m)$ with $\ch(\pi_G)\in H^0(\ctorus{n},\Q)$ having a faithful $G$-action $\rho_G:G\acts\pi_G$.
\end{prop}
\begin{proof}
We show that $R(n,m):=N(n,m)+n\geq n$ satisfies the statement where $N(n,m)\geq m$ is from Lemma~\ref{lem:numberTheory}.
Pick $G\in\Sp_{n,m}$ arbitrarily and let $d\in\Np$ be from \eqref{eq:typeS}. 
Let $\rho_\mathrm{f}:G\acts \pi_\mathrm{f}$ be the faithful action, $\rho_\mathrm{l}:G\acts \pi_\mathrm{l}$ be the action on a line bundle provided by Proposition~\ref{prop:actionOnG}. 
Let $A$ be the multiset of cardinality $N(n,m)$ produced by Lemma~\ref{lem:numberTheory} for $\delta_k=d^{2k} k!$, and let $B$ be the multiset obtained from $A$ by reducing the multiplicity of $1$ exactly by $m$. 
Now using that $\ch$ is a ring morphism, $H^{2k}(\ctorus{2n},\Z)=0$ for $k>n$ and the convention $0^0=1$, we have 
\[\chi:=
-\ch(\bigoplus_{a\in A} \pi_\mathrm{l}^{\tensor a})=
-\sum_{a\in A} \ch(\pi_\mathrm{l})^a=
-\sum_{a\in A} \exp(a c_1(\pi_\mathrm{l}))=
- \sum_{k=0}^n\sum_{a\in A}  \frac{a^k}{k!} c_1(\pi_\mathrm{l})^k
\in \bigoplus_{k=0}^n d^{2k}\Gamma_n^k
\]
because $\sum_{a\in A}a^k \in d^{2k}k!\Z$ from the choice of $A$ by Lemma~\ref{lem:numberTheory} and $c_1(\pi_\mathrm{l})^k\in\Gamma_n^k$ from Proposition~\ref{prop:actionOnG}. 
So Proposition~\ref{prop:actionChernClass} is applicable and gives an action $\rho_\chi:G\acts \pi_\chi$. Define
\begin{equation}
\pi_{G}:= 
\pi_\chi \oplus \pi_\mathrm{f} \oplus  \bigoplus_{b\in B}\pi_\mathrm{l}^{\tensor b} = 
\pi_\chi \oplus  \bigoplus_{a\in A}\pi_\mathrm{l}^{\tensor a},
\label{eq:vectorBundleTrivialC}
\end{equation}
a vector bundle of rank $n+N(n,m)=R(n,m)$ where the last quality of \eqref{eq:vectorBundleTrivialC} follows from $\pi_\mathrm{f}=\pi_\mathrm{l}^{\oplus m}$ from Proposition~\ref{prop:actionOnG} and the definitions of $A$ and $B$. 
Since $\baseAction{\rho_\chi}=\baseAction{\rho_\mathrm{f}}=\baseAction{ \rho_\mathrm{l}}=\baseAction{\rho_G}$ from \eqref{eq:actionDiagramOfG} by Proposition~\ref{prop:actionOnG},~\ref{prop:actionChernClass}, 
the actions $\rho_\chi$, $\rho_\mathrm{f}$, $\rho_\mathrm{l}$ induce a natural action on $\pi_G$ which is faithful because $\rho_\mathrm{f}$ is.
Finally note that 
$\ch(\pi_G)= 
\ch(\pi_\chi) + \ch(\bigoplus_{a\in A} \pi_\mathrm{l}^{\tensor a}) = 
\ch(\pi_\chi) -\chi \in H^0(\torus{2n},\Q)$
by Proposition~\ref{prop:actionChernClass} as required.
\end{proof}

\begin{rem}
Considering the exact bundle $\pi_\mathrm{l}$ produced by the proof of Proposition~\ref{prop:actionOnG} and noticing $\xi_d=\xi_1^{\tensor d}$ and $c_1(\xi_1^{\extTensor n})^k\in k! \Gamma_n^k$, we could have taken $\delta_k=d^k$ in the proof above instead.
\end{rem}

\begin{thm}\label{thm:main}
Question~\ref{q:main} has an affirmative answer for $\Sp_{I}^r$ for every finite subset $I\subset \Np^2$ and $r\in\Np$, 
more concretely there exists a compact connected smooth manifold $X_I$ such that for every $r\in\Np$, $\Diff(X_I^r)$ contains every element of $\Sp_I^r$ defined at \eqref{eq:SI}.
\end{thm}
\begin{proof}
We claim that $X_I:=\prod_{(n,m)\in I} \torus{2n}\times\U(R(n,m))$ satisfies the statement where $R(n,m)$ is given by Proposition~\ref{prop:actonOnTrivialChernClass}.
Indeed, pick and arbitrary $S\in \Sp_I^r$. By definition $S$ is a subgroup of some $\bigoplus_{k=1}^r G_k$ where $G_k\in \bigcup_{(n,m)\in I}\Sp_{n,m}$. 
Thus it is enough to show that $G\in\Diff(\torus{2n}\times \U(R(n,m)))$ for $G\in\Sp_{n,m}$.
To see this, apply Proposition~\ref{prop:actonOnTrivialChernClass} for $G\in \Sp_{n,m}$ to get a faithful action $\rho_G: G\acts \pi_{G}$. 
By forgetting the complex structure on $\pi_G$, we may consider it as a smooth real vector bundle over $\torus{2n}$. 
By construction, $\pi_{G}$ is a vector bundle over $\torus{2n}$ of rank $R(n,m)\geq n=\tfrac{1}{2}\dim_\R(\torus{2n})$ with $\ch(\pi_{G})\in  H^0(\torus{2n},\Q)$, and $K^{0}(\torus{2n})\isom\Z^{2^{2n-1}}$ is torsion free. 
Thus Lemma~\ref{lem:main} applies and gives precisely the containment needed.
\end{proof}

\begin{proof}[Proof of Theorem~\ref{thm:Heisenberg}/\ref{thm:specialGroups}]
Both of these statements follow from Theorem~\ref{thm:main} and  Lemma~\ref{lem:Heisenberg}/\ref{lem:Suzuki}.
\end{proof}

\begin{rem}\label{rem:ghys}
The simplest special case $I=\{(1,1)\}$ with $r=1$ of Theorem~\ref{thm:main} itself gives a counterexample of the conjecture of Ghys mentioned at \S\ref{sec:intro} which is worth considering explicitly. This counterexample essentially appeared in \cite{CsPSz} and \cite{Riera} but in a much less explicit form.
In this case, the rank $2$ vector bundle $\pi_G=\pi_\mathrm{f}\oplus \pi_\mathrm{f}^{\tensor -1}$ over $\torus{2}$ works at \eqref{eq:vectorBundleTrivialC} (more precisely, $\pi_\mathrm{f}=\pi_\mathrm{l}$, $B=\{-1\}$, but $\pi_\chi$ is the trivial line bundle, hence can be omitted). 
For this $\pi_G$ (considered as a smooth real vector bundle) we can explicitly find two linearly independent global sections (in the form of two power functions whose branch cuts are different with suitable weights to assure continuity at the cuts). 
For any $G\in \Sp_{1,1}$, the orthonormal global sections produced by the unitary frame bundle at Lemma~\ref{lem:main} actually give a faithful $G$-action on $\torus{2}\times \SU(2)$, the special unitary group as expected from \cite[Lemma 2.2]{Riera}. 

All in all, for any $E(d,j)$ from \eqref{eq:Edi} a concrete embedding $\rho:E(d,j)\to\Diff(\torus{2}\times \SU(2))$ is given by
\begin{equation}
\rho(\alpha):(\tau,\theta,S)\mapsto (\mu_d \tau,\theta,A_{\tau,\theta} S), \quad
\rho(\beta):(\tau,\theta,S)\mapsto (\tau,\mu_d \theta, B_{\tau,\theta} S), \quad
\rho(\gamma):(\tau,\theta, S)\mapsto (\tau,\theta, C_{\tau,\theta} S)
\label{eq:counterexample}
\end{equation}
where $A_{\tau,\theta}, B_{\tau,\theta},C_{\tau,\theta}\in\SU(2)$ are defined  depending on $\arg(\theta)\in[0,2\pi)$ as
\begin{center}
\begin{tabular}{c|ccc}
$\arg(\theta)\in$ & $A_{\tau,\theta}$ & $B_{\tau,\theta}$ & $C_{\tau,\theta}$ \\ \hline
$[0,\frac{\pi}{d})$ & $M^j$ & $Q_{\tau,\theta}  T_\tau^{-1} M$ & $M^d$ \\ 
$[\frac{\pi}{d},2\frac{\pi}{d})$ & $M^j$ & $T_\tau^{d-1} M$ & $M^d$ \\
$[2\frac{\pi}{d},3\frac{\pi}{d})$ & $Q_{\tau,\theta} M^j Q_{\tau,\theta}^{-1}$ & $ T_\tau^{d-1} M Q_{\tau,\theta}^{-1}$ & $Q_{\tau,\theta} M^d Q_{\tau,\theta}^{-1}$ \\
$[3\frac{\pi}{d},2\pi)$ & $M^j$ & $T_\tau^{-1} M$ & $M^d$ \\
\end{tabular}
\end{center}
for the following special unitary matrices
\[T_\tau=\begin{pmatrix}\tau & 0 \\0 & \tau^{-1}\end{pmatrix},\quad
M=\begin{pmatrix}\mu_{d^2} & 0 \\ 0 & \mu_{d^2}^{-1} \end{pmatrix},\quad
Q_{\tau,\theta}=\frac{1}{4\theta^{2d}}
\begin{pmatrix}\tau^d (\theta^d+1) & \theta^d-1 \\ \tau^d (\theta^d-1) & \theta^d+1\end{pmatrix}
\begin{pmatrix}\theta^d+1 & \tau^{-d}(\theta^d-1) \\ \theta^d-1 & \tau^{-d}(\theta^d+1)\end{pmatrix}.\]
Since this class contains every extra-special (hence non-abelian) group of order $p^3$, $\Diff(\torus{2}\times \SU(2))$ is not Jordan.
On the other hand, note that the validity of the embeddings \eqref{eq:counterexample} can be verified (independently of the whole paper) using plain matrix algebra checking the relations  $Q_{\mu_d\tau,\theta}=Q_{\tau,\theta}=Q_{\tau,\mu_d\theta}$, 
$Q_{\tau,1}=T_\tau^{d}$, $Q_{\tau,\mu_{2d}}=1$, $T_{\mu_d \tau}=M^d T_\tau$. 
\end{rem}

\printbibliography

\bigskip
\noindent
\textsc{Alfréd Rényi Institute of Mathematics}, Reáltanoda u. 13–15,
H–1053, Budapest, Hungary 
\&
\textsc{Department of Mathematics, Central European University}, Nádor u. 9, H-1051 Budapest, Hungary

E-mail address: \email{szabo.david@renyi.hu}
\end{document}